\documentclass[11pt]{amsart}
\usepackage{amssymb,amscd,amsthm,amsxtra}
\usepackage{dsfont,latexsym}
\usepackage{verbatim}
\usepackage{mathrsfs}
\renewcommand{\mathcal}{\mathscr}

\newtheorem{theorem}{Theorem}[section]
\newtheorem{lemma}[theorem]{Lemma}
\newtheorem{proposition}[theorem]{Proposition}

\newtheorem{definition}[theorem]{Definition}
\newtheorem{remark}[theorem]{Remark}

\renewcommand{\geq}{\geqslant}
\renewcommand{\ge}{\geqslant}
\renewcommand{\leq}{\leqslant}
\renewcommand{\le}{\leqslant}
\renewcommand{\phi}{\varphi}

\def\R{{\mathbb R} }
\def\N{{\mathbb N} }

\def\EN{{\mathcal E} }

\numberwithin{equation}{section}

\title[Monotonicity results]{Some monotonicity results
\\ for minimizers in the calculus of variations}

\author{Ovidiu Savin and Enrico Valdinoci}

\begin{document}

\begin{abstract}
We obtain monotonicity properties for minima and stable solutions of general energy 
functionals of the type 
$$ \int F(\nabla u, u, x) \, dx  $$
under the assumption that a certain integral grows at most 
quadratically at infinity. As a consequence we obtain several rigidity 
results of global solutions in low dimensions.
\end{abstract}

\maketitle

\section{Introduction}

In this paper we deal with monotonicity properties of minima
for quite 
general energy 
functionals of the type 
\begin{equation}\label{functional}
\int_\Omega F(\nabla u, u, x) \, dx  \end{equation}
where $\Omega$ is a domain of $\R^n$. 
These monotonicity properties are often used for the classification of global minimizers, and therefore play a key role in the regularity theory in the calculus of variations, see for example the case of minimal surfaces theory \cite{giusti}, free boundary problems \cite{AlC}, \cite{CJK}, \cite{DS} phase transitions \cite{AC} etc. Some applications of our results are given in Section \ref{s2}.  

We consider the case when the domain $\Omega$ and the functional $F$ are 
invariant under translations in a number of directions $e_k$,..,$e_n$, and we are interested in monotonicity properties of energy minimizers (or stable solutions) in the class of functions obtained by piecewise Lipschitz domain deformations in these $e_k,..,e_n$ directions. Our main result states that, under rather mild assumptions on $F$, if there exists a constant $C>0$ such that for all large $R$, a minimizer $u$ satisfies
$$\int_{\Omega \cap B_R} |D^2_pF \, (\nabla u,u,x)|\, |\nabla 
u|^2\,dx\le CR^2,$$
then $u$ is one-dimensional in each subspace generated by $e_k,..,e_n$ (see Theorem \ref{t3} for the precise statement).  

The general approach to obtain such rigidity results (see for instance the case of minimal surfaces) is to apply the stability inequality (see \eqref{stable eq}) to a suitable cutoff function. However, this approach becomes often difficult to implement. An example occurs when the functional $F$ becomes singular near $\partial \Omega$ as in the case of {\it $s$-nonlocal minimal surfaces}, $s \in (0,1)$ (see \cite{CRS}) when the energy functional has the form 
$$F(\nabla u,u,x)=|\nabla u|^2 x_1^{1-s}, \quad \Omega=\{x_1>0\}.$$
Then the stability inequality does not have a simple form due to the fact that integrations by parts are difficult to handle. 

We prove our results inspired by the simple method developed in \cite{CS} where we studied global nonlocal minimal surfaces in two dimensions. The main idea is to avoid the precise form of the stability inequality and just compare the energies of $u$ and a translation of itself. In this way we can deal with rather general situations of energy functionals and also consider minimizers directly in the natural class of functions obtained by domain deformations.

We describe briefly the strategy below. We compare the energies of $u$ 
and $$\max\{u(x), 
u(x+te_n)\},$$ 
and for this we need to modify this comparison function at infinity so that it 
becomes a 
compact perturbation of $u$. The growth condition in the integral above guarantees that the difference between the energy
of the perturbed function and the energy of~$u$ can be made arbitrarily 
small. 
On the other hand, if~$u$ is not monotone in the $e_n$ direction then we can modify locally 
the comparison function above and decrease its energy by a
small fixed amount, and this contradicts the minimality of~$u$.

The paper is organized as follows. In Section \ref{s2} we state our main theorems 
and in Section \ref{9ss} we provide some concrete applications of our results.
The main ingredients of the proofs are given in Section \ref{0055} where we perform the local analysis, and in Section \ref{s3}
where we estimate the energy of the perturbations at infinity.
The proofs will be completed in Sections \ref{33s}
and \ref{44s}. In Section \ref{55s} we discuss an explicit 1D example to illustrate better the notion of minimizer and stability in the class of piecewise Lipschitz deformations. 
Finally in Section \ref{details} we prove several remarks pointed out throughout the paper.

\section{Main results}\label{s2}
We consider energy functionals as in \eqref{functional} in the case when the domain $\Omega$ and the functional $F$ are 
invariant under translations in the $e_n$-direction, that is
$$\Omega=  \mathcal{U} \times \R, \quad \quad \mathcal {U} \subseteq 
\R^{n-1},$$
and $F$ does not depend on the $x_n$-coordinate.

Points in $\Omega$ are denoted by $x=(x',x_n)\in \mathcal{U} \times 
\R$. We assume that the functional $F$ is convex in 
with respect to the first variable. Precisely, we suppose that
\begin{equation}\label{H1}
F=F(p,z,x')\in C(\R^n\times \R \times \mathcal{U}),
\end{equation}
and for any $(z,x')\in\R\times {\mathcal{U}}$, $F$ is $C^2$ and 
uniformly convex in $p$ at all $p$ with $p_n \ne 0$.

Furthermore, we assume that~$F_{pp}=D^2_pF$ satisfies the natural growth 
condition
\begin{equation}\label{H2}
|F_{pp}(p+q,z,x')|\le C \,  |F_{pp} (p,z,x')|
\end{equation}
for some $C>0$, and any $p$, $q\in \R^n$ with $|q|\le |p_n|/2$.

For any $R>0$, we introduce the energy
functional~$\EN_R$ 
defined by
$$ \EN_R (u):=\int_{\Omega \cap B_R} F(\nabla u(x),u(x),x')\,dx.$$
We study monotonicity properties of suitable minimal or stable solutions for the energy $\EN$ 
among perturbations which are obtained by piecewise domain deformations in the 
$e_n$-direction. 
For this we introduce the following notation:

\begin{definition}\label{d1}
We say that $v$ is an $e_n$-Lipschitz deformation of $u$ in $B_R$
if there exists a Lipschitz function $\psi$ with compact support in 
$B_R$, and $\|\psi_n\|_{L^\infty(\R^n)} <1$ such that
$$v(x)=u(x+\psi(x)e_n)\qquad\forall x\in\Omega.$$
\end{definition}

In the notation of Definition~\ref{d1}, we have that if $u$ is (locally) 
Lipschitz then $v$ 
is (locally) Lipschitz as well.

\begin{definition}\label{d2}
Let $u\in C^{0,1}(\Omega)$. We say that $v\in C^{0,1}(\Omega)$ is a 
piecewise 
$e_n$-Lipschitz deformation of $u$ in $B_R$ and write $$v \in D_R(u)$$
if there exist a finite number $v^{(1)}$,..., $v^{(m)}$ of 
$e_n$-Lipschitz 
deformations of $u$ in $B_R$ such that $$v(x)=v^{(i)}(x) \quad \quad 
\mbox{for some $i$ (depending on $x\in\Omega$).}$$  
Also, if all $v^{(i)}$ satisfy
$$v^{(i)}(x)=u(x+\psi^{(i)}(x)e_n) \quad \mbox{with} \quad 
\|\psi^{(i)}\|_{C^{0,1}(\Omega)} \le \delta$$
for some $\delta>0$, we write
$$v \in D^\delta_R(u).$$
\end{definition}

We list some elementary properties that follow easily 
from Definition~\ref{d2}:
\begin{equation}\begin{split}\label{765ty}
&v, w \in  D^\delta_R(u) \quad\Rightarrow \quad\min \{v,w\}, \max 
\{v,w \} 
\in 
D^\delta_R(u);\\ 
&
v\in D^\delta_R(u), \quad w\in D^\delta_R(v)\quad\Rightarrow \quad w\in 
D^{3\delta}_R(u);\\
& v \in D^\delta_R(u)\quad\Rightarrow\quad 
\|v-u\|_{L^\infty(\Omega)} \le C \delta \|u\|_{C^{0,1}(\Omega)};\\
& v \in D^\delta_R(u), \quad
u \in C^{1,1}(\Omega) \quad\Rightarrow\quad
\|v-u\|_{C^{0,1}(\Omega)} \le C \delta \|u\|_{C^{1,1}(\Omega)}.
\end{split}\end{equation}

\begin{definition}\label{d3}
We say that~$u\in C^{0,1}(\Omega)$ is an $e_n$-minimizer for~$\EN$ if 
for any $R>0$ we have
that~$\EN_R(u)$ is finite and 
$$\EN_R(u) \le \EN_R(v), \quad \quad \forall v \in D_{R}(u).$$
\end{definition}

\begin{remark}\label{we} {\rm The standard definition in the calculus of
variation consists in saying that $u$ is              
a classical
minimizer for $\EN$ if it minimizes the energy with respect
to compact deformations of the graph of $u$ in the vertical direction ($e_{n+1}$-direction) that is:
$$\EN_R(u) \le \EN_R(u+\varphi)  $$
for any Lipschitz~$\varphi$ with compact support in $\Omega \cap B_R$.
We observe that
when $\Omega=\R^n$, $e_n$-minimality is a weaker 
condition than classical minimality.

For example any function which is constant in the $e_n$-direction is 
always an $e_n$-minimizer, but not necessarily  a classical minimizer. 
}\end{remark}

Our first general monotonicity result is the following:

\begin{theorem}\label{GENERAL}
Let~$u \in C^1(\Omega)$ be an $e_n$-minimizer for the energy $\EN$ with $F$
satisfying~\eqref{H1} and~\eqref{H2}. 

If there exists $C>0$ such that for all large $R$
\begin{equation}\label{EE}
\int_{\Omega \cap B_R} |F_{pp} \, (\nabla u,u,x')|\, |\nabla 
u|^2\,dx\le CR^2,
\end{equation}
then~$u$ is monotone on each line in the $e_n$-direction, i.e., for 
any~$\bar x\in\Omega$, either~$u_n(\bar x+te_n)\ge0$ or~$u_n(\bar x+te_n)\le 0$
for any~$t\in\R$.
\end{theorem}

\begin{remark}\label{R1}
{\rm If a continuous function $u$
is monotone on each line in $\R^n$ then it 
is one-dimensional,
that is $u=f(x \cdot \xi)$ for some function 
$f:\R \to \R$ and some unit direction $\xi.$ 
See Section~\ref{details} for a
proof.}\end{remark}

Our second theorem is a version of 
Theorem~\ref{GENERAL} for stable critical 
points of the energy instead of $e_n$-minimizers. 
The stability condition we use involves the second variation of $\EN$ 
for deformations of $u$ in the $e_n$-direction as well as in the vertical 
$e_{n+1}$-direction. 
The precise definition is the following:

\begin{definition} \label{d4}
We say that $w$ is a piecewise Lipschitz 
deformation 
of $u$ in the $\{ e_n, e_{n+1}\}$-directions and write $$w\in \mathcal 
D_R^\delta(u)$$ if 
$$w=v+\varphi \quad \mbox {with} \quad v \in D^\delta_R(u)
\quad \mbox {and} \quad |\varphi|_{C^{0,1}(\Omega)} \le \delta$$
for some Lipschitz function $\varphi$ with compact support in $\Omega \cap B_R$. 
\end{definition}

We remark that here the vertical perturbations $\varphi$ have compact 
support in $\Omega \cap B_R$ whereas the $e_n$-deformations $\psi^{(i)}$
in Definitions~\ref{d1} and~\ref{d2} 
have compact support in $B_R$ (i.e., if~$x\in B_R$ with~$x'\in
\partial \mathcal{U}$ then~$\varphi(x)=0$ but~$\psi^{(i)}(x)$ may
be different from~$0$).

\begin{definition}\label{d5}
We say that $u$ is a $\{e_n, e_{n+1}\}$-stable solution for $\EN$ if for 
any $R>0$ and $\epsilon>0$
there exists $\delta>0$ depending on $R$, $\epsilon$ 
and $u$ such that for all $t\in(0,\delta)$ we have that~$\EN_R(u)$ is finite and
\begin{equation}\label{55}
\EN_R(w)-\EN_R(u) \ge -\epsilon t^2, \qquad \forall w \in 
\mathcal D^t_R(u).\end{equation}
\end{definition}

We point out that classical minimality (see Remark~\ref{we})
implies $\{e_n, e_{n+1}\}$-stablity
(on the other hand, $e_n$-minimality and
$\{e_n, e_{n+1}\}$-stablity do not imply each other in general).
Also,
since we allow perturbations in the $e_{n+1}$-direction in
Definition~\ref{d5}, 
then any $\{e_n, e_{n+1}\}$-stable solution is a critical point
of the energy functional. 

\begin{remark}\label{stable} {\rm In the calculus of variation, it is customary to consider
stable solutions of partial differential equations. Classically, a solution (i.e., a critical
point of the energy functional) is said to be stable if
\begin{equation}\label{stable eq}
\liminf_{t\to0} \frac{\EN_R (u+t \phi)-\EN_R(u) }{ t^2} \ge 0 \end{equation}
for any Lipschitz function~$\phi$ supported in~$B_R$.

If~$\Omega=\R^n$, $F\in C^2$ and~$u\in C^2$, this classical notion of stability
is equivalent to the notion
of~$\{e_n,e_{n+1}\}$-stable solution (for the proof of this, see Section~\ref{details}).
}\end{remark}

In the framework given by Definition~\ref{d5}
we prove the following result.

\begin{theorem}\label{GENERAL2}
Let~$u\in C^{0,1}(\Omega)$ be a $\{e_n,e_{n+1}\}$-stable solution and assume $F 
\in C^3(\R^2\times\R\times{\mathcal{U}})$ satisfies \eqref{H2}.

If the growth condition \eqref{EE} holds 
then~$u$ is monotone in the~$e_n$-direction,
i.e. either~$u_n\ge0$ or~$u_n\le0$ in $\Omega$.
\end{theorem}

We observe that the hypotheses in the two theorems above are slightly different and
the thesis of
Theorem~\ref{GENERAL}
is weaker than the one of 
Theorem~\ref{GENERAL2}
since in Theorem~\ref{GENERAL} we do not say that $u_n(x)$ has 
the same sign
for all $x$, but only
that, fixed $x$, $u_n(x+te_n)$ has the same sign for any $t$.

Our last theorem deals with $\{e_k,..., e_n\}$-stable solutions, that is, in the definition of stability we allow small piecewise Lipschitz deformations in the $e_k, .., e_n$-directions rather than only the $e_n$-direction or $\{e_n,e_{n+1}\}$-direction (see Definition \ref{d5.5} for a precise statement).  

\begin{theorem}\label{t3}
Assume that
$$\Omega=  \mathcal{U} \times \R^{n-k+1}, \quad \quad \mathcal {U} 
\subseteq
\R^{k-1},$$
$F$ does not depend on the $x_k$, ..., $x_n$ coordinates, ~$F$ satisfies \eqref{H2}
and that~$F \in C^3$ at all $p$ with $(p_k,...,p_n) \ne(0,..,0)$. 

If $u \in C^1(\Omega)$ is $\{e_k,..., e_n\}$-stable and the growth condition~\eqref{EE} holds, then $u$ is one-dimensional in any subspace 
generated by $\{e_k,...,e_n\}$.
\end{theorem}

The theorem concludes that for each $(x_1,...,x_{k-1})\in\mathcal{U}$, $u(x)$ is one-dimensional (see Remark \ref{R1}) in the remaining variables $(x_k,\ldots,x_n)$. Of course, when $k=n$ the statement becomes trivial.

We point out that the hypothesis above on $\{e_k,..., e_n\}$-stability for $u$ is in general easily satisfied by critical points of $\EN$ which are monotone in the $e_n$ direction, and in fact such critical points are $\{e_k,..., e_n\}$-minimizers.

We conclude this section with several remarks on the theorems above.

\begin{remark}{\rm The results provided in this paper are in fact even more general:
we did not attempt to give the most general conditions
possible but rather to emphasize the method of proof
(further generalizations will be outlined in subsequent remarks
and some of these generalizations turn out to be important in the concrete applications).
For instance, we observe that the functional in~\eqref{functional}
may be generalized to
\begin{equation}\label{functional2}
\int_\Omega F(\nabla u, u, x') \, dx+\int_{\partial \Omega} G(u,x')\,d\mathcal{H}^{n-1}, \end{equation}    
where~$G$ satisfies the same regularity assumptions as~$F$.
The proofs in this case are affected only by minor, obvious modifications.
}\end{remark}

\begin{remark}\label{log}{\rm Condition \eqref{EE}
may be weakened by allowing logaritmic corrections too.
For instance, the right hand side of \eqref{EE} may be replaced by
$$ CR^2 \log R$$
or by
$$ CR^2 (\log R)(\log\log R).$$
More generally, one can define $\ell_0(R):=R$
and recursively $$
\ell_k(R):= \log (\ell_{k-1}(R))=
\underbrace{\log\circ\dots\circ\log}_{\text{$k$ times}}R
$$ for any $k\in\N$,
$k\ge 1$. Let also
\begin{equation}\label{PI}
\pi_k(R):=\prod_{j=0}^k \ell_j(R).\end{equation}
Then, instead of \eqref{EE}, one may take the weaker condition
\begin{equation}\label{EE weak}
\int_{\Omega\cap B_R}|F_{pp}(\nabla u,u,x')|\,|\nabla u|^2\,dx
\le CR \,\pi_k(R),
\end{equation}
for a given $k\in\N$.
For the proof of this fact, see Section \ref{details}
(notice that \eqref{EE weak} boils down to \eqref{EE}
if $k=0$). An energy growth with a logaritmic correction of the
type~$CR^2\log R$
was also considered in \cite{mos} in the case of semilinear equations.
}\end{remark}

\begin{remark} \label{ABS} {\rm
At first glance, Definition \ref{d2} 
may look unnecessarily complicated, since one may think that
Definition \ref{d1} suffices for Theorem \ref{GENERAL}. That is,
one may think that if $u$ minimizes the energy
with respect to any $e_n$-Lipschitz deformation and \eqref{EE}
is satisfied, then $u$ must possess some kind of monotonicity.
However this is not the case, as we show by an example in Section \ref{details}.
}\end{remark}

\section{Applications}\label{9ss}

Below we present some direct applications of our results and obtain several rigidity results of global solutions in low dimensions. We remark however that our theorems do not give in general the optimal dimension for these rigidity results. 
\subsection{De Giorgi's conjecture}
As a first application,
we obtain a classical one-dimensional symmetry property
related to a conjecture of De Giorgi (see~\cite{DG}):

\begin{theorem}[\cite{GG, BCN, AC, AAC}]
\label{DG}
Let~$f\in C(\R)$ and~$u\in C^2(\R^n)\cap L^\infty(\R^n)$ be a solution 
of~$-\Delta u+f(u)=0$ in the whole of~$\R^n$.

Suppose that:
\begin{itemize}
\item either $n=2$ and $u$ is stable (according to the notation recalled 
in Remark~\ref{stable}),
\item or $n=3$ and~$u_3>0$.
\end{itemize}
Then~$u$ is one-dimensional.
\end{theorem}

\begin{proof}
We let~$\tilde F$ be a primitive of~$f$ and we define
$$ F(p,z,x):=\frac12|p|^2+\tilde F(z).$$
Then,~$F$ is clearly convex in~$p$ and it satisfies~\eqref{H2}.
It also satisfies~\eqref{EE}: when~$n=2$ this simply follows from the fact
that~$|B_R|\le CR^2$, and when~$n=3$ it is a consequence of
Theorem~5.2 in~\cite{AAC}. 

Now we apply
Theorem~\ref{GENERAL2} and obtain that~$u$ is one-dimensional.
\end{proof}

 We stress that the proof of
Theorem~\ref{DG} that we give here is based
on domain perturbations and it does not use some of
the basic ingredients exploited in the existing literature:
e.g., differently from \cite{BCN, AC, AAC}, it does not use
any Liouville type result, differently from \cite{GG}
it does not use the Ekeland's variational principle,
differently from \cite{Far} it makes no use of any complex
structure, differently from \cite{S} no costruction of barriers
is needed, and differently from \cite{FarH, FSV}
no geometric Poincar\'e inequality is exploited.

\begin{remark}\label{6fsv}{\rm The one-dimensional
results related to the Conjecture of De Giorgi
in dimensions~$2$ and~$3$ may be extended to a very broad class
of operators and nonlinearities: see Theorems~1.1 and~1.2
in~\cite{FSV}. We remark that our Theorem~\ref{GENERAL2}
also implies Theorems~1.1 and~1.2
in~\cite{FSV} (at least in case of smooth nonlinearities;
for a proof of this fact see Section~\ref{details}).
}\end{remark}

\subsection{Fractional De Giorgi conjecture}
The one-dimensional symmetry of Theorem~\ref{DG} has a 
counterpart in the fractional Laplace framework, that may be also obtained
as a consequence of the results of this paper:

\begin{theorem}[\cite{Sola, CSire, CCinti, CCinti2}]
\label{DGs}
Let~$s\in(0,1)$, $f\in C(\R)$ and~$u\in C^2(\R^n)\cap L^\infty(\R^n)$ be 
a solution
of~$(-\Delta)^s u+f(u)=0$ in the whole of~$\R^n$.

Suppose that:
\begin{itemize}
\item either $n=2$ and $u$ is stable (according to the notation recalled
in Remark~\ref{stable}),
\item or $n=3$, $s\in[1/2,1)$ and~$u_3>0$.
\end{itemize}
Then~$u$ is one-dimensional.
\end{theorem}

\begin{proof}
We use the extension result in~\cite{Silv}
and therefore we reduce this problem
to an energy functional 
in~$(0,+\infty)\times\R^n$ as the one in~\eqref{functional2}
with~$${\mathcal{U}}:=(0,+\infty)\times \R^{n-1}, \quad
F(p,z,x):=x_{1}^{1-2s} |p|^2, \quad \quad G(z,x):=\tilde F(z),$$
where~$\tilde F$
is a primitive of~$f$ (notice that~$n$ in Theorem~\ref{GENERAL2}
must be replaced by~$n+1$ for this application). Then,
the desired energy growth follows from~\cite{CCinti, CCinti2}, according to which
$$ \EN_R(u)\le \left\{\begin{matrix}
CR^{n-\min\{2s,1\}} & {\mbox{ if }}s\ne 1/2,\\
CR^{n-1}\log R & {\mbox{ if }}s=1/2.
\end{matrix}
\right.$$
Therefore, \eqref{EE} is satisfied when~$n=2$ and also when~$n=3$
and~$s\in( 1/2,1)$
(on the other hand, when~$n=3$ and~$s=1/2$, \eqref{EE weak} is satisfied
and one has to make use of Remark~\ref{log}).
Thus
we obtain Theorem~\ref{DGs}
as a consequence of Theorem~\ref{GENERAL2}
and Remark~\ref{R1}.
\end{proof}

\subsection{Minimal surfaces}

Minimal surfaces in $\R^n$ can be thought as boundaries of sets $E \subset \R^n$ that minimize the BV-norm or the perimiter (see \cite{giusti}) 
$$ Per(\chi_E):=\int |D\chi_E|\,dx.$$
Although the functional $F$ does not satisfy precisely the conditions of our theorems, the methods of proof of the next two sections easily apply to this case as well. Then condition \eqref{EE} reads
$$Per(\chi_E, B_R)=\int_{B_R}   |D\chi_E|\,dx \le C R^2.$$
On the other hand the perimeter of a minimal surface in $B_R$ is bounded by the surface area of $\partial B_R$, that is $C R^{n-1}$, 
hence the only global minimal surfaces in $\R^3$ are the hyperplanes (one-dimensional). 

\subsection{Nonlocal minimal surfaces}
As mentioned in the Introduction we discuss a result
on nonlocal perimeters which was the original 
motivation for the techniques developed in this paper, see~\cite{CS}. 

Given a bounded domain~$\Omega\subset\R^n$, the minimization of 
the following functional was introduced in~\cite{CRS}:
$$ {\rm Per}_s(E,\Omega):= L(E\cap\Omega,\R^n \setminus E) + L(E 
\setminus\Omega,\Omega\setminus E),$$
where $s\in(0,1)$ and for any disjoint measurable sets~$A$ and~$B$,
$$ L(A,B) :=\int_A\int_B\frac{dx\,dy}{|x-y|^{n+s}}.$$
The regularity of $s$-minimal surfaces (i.e. of the boundary of
a set~$E$ which minimizes~$ {\rm Per}_s(\cdot,\Omega)$
among all the measurable sets that agree with~$E$ outside $\Omega$)
and of $s$-minimal cones (i.e. of $s$-minimal surfaces~$E$ such 
that are invariant under dilations)
has been studied in some recent papers, such as~\cite{CRS, CV2, CS, Bego}.
In particular, a complete regularity theory holds in the plane, according 
to the following result, that may also be obtained as a byproduct
of the results in this paper:

\begin{theorem}[\cite{CS}]\label{SVs}
If $E$ is an $s$-minimal cone in~$\R^2$, then $E$ is a 
half-plane.
\end{theorem}

\begin{proof}
By the extension result in Section~7 of~\cite{CRS}, we reduce
the problem to a variational energy in~$(0,+\infty)\times\R^2$,
with~$$F(p,z,x):=x_1^{1-s} |p|^2,$$ for a minimizer homogenous of degree $0$. Then, \eqref{EE} easily follows in dimension $n=2$, and so we may use again
Theorem~\ref{GENERAL2}.
\end{proof}

\subsection{Two-phase free boundary problem}
This classical free boundary problem (see \cite{AC}, \cite{ACF}) consists in minimizing the energy
$$\int |\nabla u|^2 dx + |\{u >0\}|.$$
In this case condition \eqref{EE} becomes
$$\int_{B_R} |\nabla u|^2 dx \le C R^2,$$
which is clearly satisfied by a Lipschitz minimizer in dimension $n=2$. In conclusion, in $\R^2$ any Lipschitz minimizer for the two-phase problem must be one-dimensional.

\subsection{Thin one-phase problem} In this free boundary problem we minimize the following energy in $\R^{n+1}_+$ (see \cite{DS})
$$\int_{\R^{n+1}_+} |\nabla u|^2 dX + \mathcal H^n(\{ u(x,0)>0\}),$$
where we denote the points in $\R^{n+1}$ by $X=(x,x_{n+1})$.

Our results imply that in dimension $n=2$, any homogenous minimizer must be one-dimensional in the $x$ variable. This follows easily from \eqref{EE} since, due to the scaling of the energy, any homogeneous minimizer must be homogenous of degree $1/2$.

\section{Local perturbations}\label{0055}

In this section we show that in general we can perturb locally $$\max\{ u(x), u(x+te_n)\}$$ into a function with lower energy.

The first lemma states that the maximum of two functions that form an 
angle at an intersection point
cannot be an $e_n$-minimizer 
for $\EN$ 
(this fact uses the strict convexity of~$F$ in the $p$ variable).

\begin{lemma}\label{l1}
Assume $0 \in \Omega$ and $u$, $v$ are $C^1$-functions such that 
\begin{equation}\label{angle}
{\mbox{$u(0)=v(0)$ and $v_n(0)<0<u_n(0)$.}}
\end{equation}
Then $g:=\max\{u,v\}$ is 
not 
an $e_n$-minimizer for $\EN$ in any ball $B_\eta$. 
\end{lemma}

\begin{remark}\label{F5}
{\rm In our setting,
the transversal intersection
described analytically by~\eqref{angle}
can be obtained whenever $u$ is not monotone on each line along the $e_n$-direction. In this case we may reduce to the case in which~$u(\bar  
x+a_1 e_n)<u(\bar x+a_2e_n)$ and~$u(\bar x+a_2e_n)>u(\bar x+a_3e_n)$,
with~$a_1<a_2<a_3$. 
Let~$c_i:=u(\bar x+a_i e_n)$. 
Then, by Sard's theorem we can find a regular value ~$c\in \big( \max\{c_1,c_3\},\, c_2\big)$ of $u$, thus we may find~$\alpha_c\in(a_1,a_2)$
and~$\beta_c\in (a_2,a_3)$ such that~$u(\bar x+\alpha_c e_n)=c=
u(\bar x+\beta_c e_n)$ and ~$u_n (\bar x+\alpha_c e_n)
> 0 >
u_n (\bar x+\beta_c e_n)$.
Then, the setting
of~\eqref{angle} is fulfilled by supposing, up to translations, that~$
\bar x+\alpha_c e_n=0$ and by taking~$v(x):=u(x+(\beta_c-\alpha_c)e_n)$.
}\end{remark}

\begin{proof}[Proof of Lemma~\ref{l1}]
Assume by contradiction that $g$ is an $e_n$-minimizer in some small 
ball $B_\eta$. We define~$F_0(p):=F(p,0,0)$, 
and we claim that we may reduce to the case in which
\begin{equation*}\label{reduce}
F_0(\nabla u(0))=F_0(\nabla v(0)).
\end{equation*}
To see this we notice that the property of minimality is not affected after subtracting a linear functional from $F$. Precisely if  $$\tilde F(p,z,x):=F(p,z,x)-p_0\cdot p,$$ and $ \tilde\EN_R$ is the associated energy functional for $\tilde F$ in $B_R$ then
$$\EN_R(f)- \tilde \EN_R(f)= \int_{B_R}p_0\cdot \nabla f \,dx
=\int_{\partial B_R}  f \, p_0\cdot\nu.
$$
That is,~$\tilde \EN_R(f)$ and~$\EN_R(f)$ only differ by a term depending on the boundary
values of~$f$. Consequently, if~$f$ is an~$e_n$-minimizer for~$\EN$, it is also an~$e_n$-minimizer
for~$\tilde\EN$.

Also, by possibly translating~$F$ in the~$z$-variable,
we may assume that $u(0)=v(0)=0$.
Now, for small~$r>0$, we consider the rescalings 
$$u_r(x):=r^{-1}u(rx), \quad   v_r(x):=r^{-1}v(rx)$$
and we define~$g_r(x):=\max\{ u_r(x),v_r(x)\}$.
Then,~$g_r$
is an $e_n$-minimizer for the rescaled functional
$$F_r(p,z,x):=F(p, rz,rx)$$
in $B_{\eta/r}$.
As $r \to 0^+$ then the following limits hold uniformly on compact sets:
\begin{equation}\label{G3}\begin{split}
& F_r \to F_0(p),
\\ & u_r (x)\to u_0(x):=\nabla u(0) \cdot x, \quad \quad \nabla u_r 
\to \nabla u_0,\\ &v_r(x) \to v_0(x):= \nabla v(0) \cdot  x, \quad \quad \nabla v_r \to 
\nabla v_0.\end{split}\end{equation}
So we let
$$g_0=\max\{u_0,v_0\}.$$ 
From the strict convexity of $F$ in the $p$ variable we see 
that~$g_0$
is not a minimizer for $F_0$. 
Indeed we first construct $h_0$, 
$$h_0:=1+ \alpha u_0 +(1-\alpha) v_0-\rho_R(x'), \quad \quad \rho_R(x'):=\max\{0,|x'|-R\}$$
for some $\alpha \in (0,1)$ small and $R$ large. Then
\begin{equation*}\label{G4}
\max\{g_0, h_0\},
\end{equation*}
coincides with $g_0$ outside $B_{R+C}$  and notice that
in $B_R$ we are cutting the graphs of two transversal linear functions by
a single one. This function has lower energy for~$F_0$
than the one of~$g_0$ provided that we choose $R$ sufficiently large. 

By using the uniform convergence in~\eqref{G3}, we see that
$$h_r:=\max\{ g_r, h_0 \},$$ 
has lower energy for $F_r$ than the one of~$g_r$.

Scaling back, we have that~$h_\star(x):=r h_r(x/r)$
has less energy for~$F$ in~$B_{r(R+C)}\subseteq B_\eta$
than the one of~$g$.
To reach a contradiction, it remains to check
that~$h_\star$ is indeed an allowed perturbation according to
Definition~\ref{d2}. This is equivalent to say that $h_r$ is a piecewise Lipschitz domain deformation of $g_r$ with the Lipschitz norm bounded by $\delta$. 

To obtain this, we use our hypothesis 
$\nabla u_0 \cdot e_n >0 >\nabla v_0 \cdot e_n$
 and the uniform convergence (in $C^1$) of $u_r$ and $v_r$ to $u_0$ respectively $v_0$.
 Then, by the Implicit Function Theorem, the part of the graph of $h_r$ where $h_0 > g_r$
  is obtained from $u_r$ by a Lipschitz domain deformation with Lipschitz norm less than $\delta$, 
  provided that $\alpha$ is chosen sufficiently small.
\end{proof}

\begin{remark}\label{F6}
{\rm
In the proof we also showed that if $u$, $v$ are $C^1$ functions with $u(0)=v(0)$ and $\nabla u(0) \ne \nabla v(0)$ then $g:=\max\{u,v\}$ is not a classical minimizer for $\EN$ in $B_\eta$. 
}
\end{remark}

The second lemma deals with perturbations for $\max\{u(x), u(x+te_n) \}$ 
(for small $t$) near a non-degenerate point on $\{u_n=0\}$.

\begin{lemma}\label{l2}
Assume that $u \in C^2(\Omega)$ is a critical point for the energy $\EN$ 
in a 
neighborhood of the origin and the functional $F \in C^2$ in a neighborhood of $(\nabla u(0),u(0),0)$. Assume that
\begin{equation}\label{NoG}
u_n (0) =0, \quad \nabla u_n(0) \ne 0\end{equation}
and let \begin{equation}\label{3.7a}
w(x):=\max\{u(x), u(x+t e_n) \}.\end{equation}
Then, for any $\eta>0$, 
there exists a 
Lipschitz function $\varphi$ with compact support in $B_\eta$ such that
$$\EN_\eta(w+t \varphi)-\EN_\eta(w) \le - c t^2  \quad \quad \mbox{for 
all 
$t$ small,}$$
for some small $c>0$ depending on $u$, $F$ and $\eta$.
\end{lemma}

\begin{proof}
Let \begin{equation}\label{3.7bis}
v(x):=\frac{u(x+te_n)-u(x)}{t}\end{equation}
and notice that
\begin{equation}\label{v-u}
\|v-u_n\|_{C^{0,1}(B_\eta)}=o(1) \quad \mbox{as $t \to 0$}.
\end{equation}
Given a Lipschitz function $g$ we use that $F \in C^2$ in the $(p,z)$ variables and obtain
$$\EN_\eta(u+t g)=\EN_\eta(u) + t L(g) + t^2 Q(g) + o(t^2)$$
with
$$L(g):=\int_{B_\eta} F_p \cdot \nabla g + F_z \, g \, dx , $$
$$Q(g):= \int_{B_\eta} G(\nabla g, g,x) \, dx =
\int_{B_\eta} (\nabla g)^TF_{pp} \nabla g + 
2 g F_{pz} \cdot \nabla g + F_{zz} g^2 \, \, dx.$$
In the integrals above the function $F$ and its derivatives are evaluated at $(\nabla u,u,x)$ and the constant in the error term $o(t^2)$ depends on $u$, $F$ and $\|g\|_{C^{0,1}(B_\eta)}$.
Since $u$ is a critical point for $\EN$ we see that if $\varphi$ has compact support in $B_\eta$ then
\begin{equation}\label{08}
\EN_\eta (u+t v^+ + t \varphi) - \EN_\eta (u + t v^+)= t^2(Q(v^++ 
\varphi)-Q(v^+)) +o(t^2).\end{equation} 
{F}rom \eqref{3.7a} and \eqref{3.7bis}, we see that
\begin{equation}\label{080}
w=u+tv^+.\end{equation} Also, 
we claim that, if~$\eta$ is sufficiently small,
\begin{equation}\label{09}
Q(v^+)-Q(u_n^+)=o(1) \ {\mbox{ and }} \
Q(v^+ +\varphi)-Q(u_n^+ +\varphi)=o(1).
\end{equation}
We prove the first relation, the second being analogous.
For this, we fix~$\mu>0$ and we define~${\mathcal{A}}_\mu:= B_\eta\cap 
\{|u_n|\le\mu\}$ and~${\mathcal{B}}_\mu:= B_\eta\cap
\{|u_n|>\mu\}$.
{F}rom~\eqref{v-u}, we have that
\begin{equation}\label{090}
\lim_{t\to0} \|v^+-u_n^+\|_{C^{0,1}({\mathcal{B}}_\mu)}=
\lim_{t\to0}\|v-u_n\|_{C^{0,1}({\mathcal{B}}_\mu)}=0.\end{equation}
On the other hand,
since~$\nabla u_n(0) \ne 0$,
for small~$\eta$ we have that the measure of~${\mathcal{A}}_\mu$
is (at most)
of the order of~$\mu$. This and~\eqref{v-u}
yield that
$$ \lim_{t\to0}|Q(v^+)-Q(u_n^+)|\le C\mu$$
and so~\eqref{09} follows since~$\mu$ can be taken arbitrarily small.

{F}rom~\eqref{09}
we see that if $\eta$ is sufficiently small
we can replace $v^+$ by $u_n^+$ in the right hand side of~\eqref{08}:
accordingly, recalling also~\eqref{080}, we obtain
\begin{equation}\label{u7}
\EN_\eta(w+t\varphi)-\EN_\eta(w)=t^2(Q(u_n^+ 
+\varphi)-Q(u_n^+))+o(t^2).\end{equation}
On the other hand $u_n$, $0$ and $G$ satisfy the hypotheses of 
Remark~\ref{F6}, hence $u_n^+$ is not a minimizer of $Q$. Thus we can 
choose 
$\varphi$ such that
$$ Q(u_n^+ +\varphi)\le Q(u_n^+)-c$$
for some small~$c>0$, possibly
depending on $u$, $F$ and $\eta$.
So, by~\eqref{u7},
$$ \EN_\eta(u+t v^+ + t \varphi) - \EN_\eta(u + t v^+) \le -\frac{c}2 
t^2$$
for all small $t$.
\end{proof}

\begin{remark}\label{r2} {\rm If $\nabla u(0) \ne 0$ then the function $w+t 
\varphi$ 
can be interpreted (via the Implicit Function Theorem)
as a Lipschitz domain deformation of $w$ in the $\nabla u(0)$-direction 
(see Definition \ref{d1}) and the $C^{0,1}$-norm of the deformation is 
bounded by $C t$. 
Notice that, in general, the $\nabla u(0)$-direction and the $e_n$-direction 
are different.
}\end{remark}  

The non-degeneracy hypothesis $\nabla u_n \ne 0$ of Lemma~\ref{l2} can 
be checked easily from Hopf lemma if $F\in C^3$ in a neighborhood of $\nabla u(0)$, as next result
points out.

\begin{lemma}\label{L3}
Assume that $u\in C^1(\Omega)$ is a critical point for $\EN$ and $F\in 
C^3$ in 
a neighborhood of $(\nabla u(0),u(0),0)$. If $u_n(0)=0$ and $u_n$ does not vanish identically in a neighborhood of $0$ then there exists a point $x_0$ close to $0$ such that $u_n(x_0)=0$, $\nabla u_n(x_0) \ne 0$. 
\end{lemma}

\begin{proof}
Since $u$ is a critical function for $\EN$ then it satisfies the 
elliptic equation
\begin{equation*}
G(D^2u,Du,u,x'):=div \, F_p (\nabla u,u,x') - 
F_z(\nabla 
u,u,x')=0.\end{equation*}
From the De Giorgi-Nash-Moser theorem and the Schauder estimates (see~\cite{GT}) it follows that if ~$u$ is locally Lipschitz and $F \in 
C^{2,\alpha}$ then ~$u \in 
C^{2,\alpha}$ and the equation above is 
satisfied there in the classical sense. If $F \in C^3$ then $G \in C^1$ hence 
by differentiating the equation in the $e_n$-direction we see that 
$v=u_n$ satisfies the linearized equation (in the viscosity sense)
$$Lv:=G_{ij} v_{ij} + G_{p_i} v_i + G_z v =0,$$
where the derivatives of $G$ are evaluated at $(D^2u,D u ,u,x')$. 
Since $v$ does not vanish identically we can apply Hopf lemma to $v$ at 
a point $x_0\in \{v=0\}$ which admits a tangent ball from either 
$\{v>0\}$ or $\{v<0\}$. 
\end{proof}

\section{Perturbations at infinity}\label{s3}

For all $R$ large we define the Lipschitz continuous function $\psi_R$ with compact support in $\R$ given by
\begin{equation}\label{12.1}
\psi_R(s) := \begin{cases}1, \quad 0 \leq s \leq \sqrt R, \\ \ \\2-\dfrac{2\log s}{\log R}, \quad \quad \sqrt R < s \leq R, \\ \ \\ 0, \quad s > R.\end{cases}\end{equation}
Notice that 
\begin{equation}\label{12.2}
\psi'_R(s) = \begin{cases}0, \quad s \in (0,\sqrt R) \cup (R,\infty), \\ \ \\  \dfrac{-2}{s\log R}, \quad s\in(\sqrt R,R).\end{cases}\end{equation}
For $0< t  \le \sqrt R/4$, we define a bi-Lipschitz change of 
coordinates:
$$x \mapsto y(x):= x+  t \psi_R (|x|) e_n$$
and let $$u_{R,t}^+(y) = u(x).$$ Notice that $u^+_{R,t}(x)$ coincides 
with $u(x-te_n)$ in $B_{\sqrt R/2}$ and with $u(x)$ outside $B_{R}$. 
Next we estimate $\EN_{R}(u^+_R)$ in terms of $\EN_{R}(u)$. We have
$$D_x y = I+A,$$ with $$A(x) = t \, \psi_R'(|x|)\begin{pmatrix}
  0 & 0 & \cdots & 0 \\
  0 & 0 & \cdots & 0 \\
  \vdots  & \vdots  & \ddots & \vdots  \\
  \frac{x_1}{|x|} & \frac{x_2}{|X|} & \cdots & \frac{x_n}{|x|}
 \end{pmatrix} $$ and $$\|A\| \leq t |\psi'_R(|x|)| \ll 1.$$
Notice that 
$$D_y x = (I+A)^{-1} = I - \frac{1}{1+trA} A.$$
We have,
$$\nabla_y u^+_R = \nabla_x u \;D_yx, \quad dy = (1+tr A) dx,$$
thus
\begin{eqnarray*} &&
\int_{\Omega \cap B_R} F(\nabla_y u^+_{R,t},u^+_{R,t},y') dy 
\\ &&\quad= \int_{\Omega \cap B_R} F \left (\nabla_x u \, \, \left (I- 
\frac{1}{1+trA} A\right ), u, x'\right) (1+tr A) dx.\end{eqnarray*}
We bound the right hand side from above by using that $|(pA)| \le |p 
\cdot e_n| /4$ which together with hypothesis \eqref{H2} for $F$ gives 
that
$$F \left (p \, \, \left (I- \frac{1}{1+trA} A\right ), z, x'\right) (1+tr A) $$ is bounded above by 
\begin{equation}\label{ab}
F(p,z,x') (1+ tr A) - F_p(p,z,t) \cdot (pA) +  C |F_{pp} (p,z,t)| 
|pA|^2.\end{equation}
By writing the same inequality for $u^-_{R,t}$ which is defined as 
$u_{R,t}^+$ with $t$ replaced by $-t$, thus $A$ is replaced by $-A$ in the formulas above, we obtain
\begin{align}\label{AB}
\nonumber
& \EN_{R}(u^+_{R,t})+\EN_{R}(u^-_{R,t})-2 \EN_{R} (u) \\
& \qquad\le C \int_{\Omega \cap B_R} |F_{pp}(\nabla u,u,x')|| \nabla 
u|^2 |A|^2 dx  \\
& \qquad
\le C \frac{t^2}{(\log R)^2} \int_{\Omega \cap \big( B_R \setminus 
B_{\sqrt R} \big)}\frac {|F_{pp}(\nabla u,u,x')|| \nabla u|^2}{|x|^2} 
dx. \nonumber
\end{align}
We denote by \begin{equation}
\label{5.2a}
a(r):= \int_{\Omega \cap B_r} |F_{pp}(\nabla u,u,x')|| 
\nabla u|^2 dx\end{equation}
and by hypothesis \eqref{EE} we know that~$a(r) \le Cr^2$. 
Then the last integral in~\eqref{AB} is controlled, in polar
coordinates, by
\begin{equation}\label{5.2b}
\int_{\sqrt R}^R a'(r) r^{-2}  dr \le a(R)R^{-2} + 2 \int_{\sqrt R}^R a(r) r^{-3} \le C \log R.\end{equation}
{F}rom \eqref{AB} and~\eqref{5.2b} we conclude that
\begin{equation}\label{MAJOR}
\limsup_{R\rightarrow+\infty}\sup_{t\in(0,\sqrt R/4)}
t^{-2}\Big(
\EN_R(u^+_{R,t})+\EN_R(u^-_{R,t})-2\EN_R(u)\Big)\le0.
\end{equation}

\section{Proofs of Theorems \ref{GENERAL} and \ref{GENERAL2}}\label{33s}

\begin{proof}[Proof of Theorem \ref{GENERAL}]
Since $u$ is an $e_n$-minimizer we know that $$\EN_{R}(u^+_{R,t}) \ge 
\EN_{R}(u).$$ This and~\eqref{MAJOR} imply that, for any fixed $t$, we 
have 
\begin{equation}\label{3.0}
\lim_{R \to +\infty} \EN_{R} (u^-_{R,t}) - \EN_{R} (u) =0.\end{equation}
Now we recall the integral formula
\begin{equation}\label{3.1}
\EN_{R} (\max\{u^-_{R,t}, u \}) + \EN_{R} (\min\{u^-_{R,t}, 
u \}) = \EN_{R} (u^-_{R,t}) + \EN_{R} (u) ,\end{equation}
and we make use of the minimality of~$u$, which implies that
\begin{equation}\label{3.2}
\EN_{R}(\min \{u^-_{R,t}, u\}) \ge \EN_{R}(u). \end{equation}
By~\eqref{3.0}, \eqref{3.1} and~\eqref{3.2} we find
\begin{equation}\label{8.1}
\lim_{R \to +\infty} \EN_{R} (v_{R,t}) - \EN_{R} (u) =0,\end{equation}
with
\begin{equation}\label{5.4a}v_{R,t}:=\max\{ u_{R,t}^-,u\}.\end{equation}
Notice that $$v_{R,t}=\max\{ u(x), u(x+te_n)\} \quad \mbox{in $B_{\sqrt R /4}$},$$ and $ v_{R,t} \in  D_R^t(u)$.

Now assume by contradiction that $u\in C^1(\Omega)$ is not monotone on a 
line in 
the $e_n$-direction. Then we can find $t>0$ so that $u(x)$, $u(x+te_n)$ 
satisfy the hypotheses of Lemma~\ref{l1} (say, at some point $x_0 \in 
\Omega$, recall Remark~\ref{F5}).
 
Thus 
we can perturb $v_{R,t}$ locally near $x_0$ into $\tilde v_{R,t}$ such that 
\begin{equation}\label{8.2}
\EN_R (\tilde v_{R,t}) \le \EN_R (v_{R,t}) - 
c\end{equation}
for some fixed $c>0$ depending only on $u$. {F}rom~\eqref{8.1}
and~\eqref{8.2} we 
contradict the 
minimality of $u$ as $R \to +\infty$.
\end{proof}

\begin{proof}[Proof of Theorem \ref{GENERAL2}]
We argue as above and use Lemma~\ref{l2} instead. Given $\epsilon>0$ we 
choose $R$ large such that
$$\EN_R(u^+_{R,t})+\EN_R(u^-_{R,t})-2\EN_R(u) \le \epsilon t^2.$$
Since $u$ is $\{e_n,e_{n+1}\}$-stable we have
$$\EN_R(w) \ge \EN_R(u) - \epsilon t^2 \quad \forall w \in \mathcal D^t_R(u),$$
for all $t$ small enough (the first relation above comes
from~\eqref{MAJOR} and the second one from Definition~\ref{d5}). Then,
using also~\eqref{3.1} and~\eqref{5.4a}, 
we obtain 
$$\EN_{R} (v_{R,t}) - \EN_{R} (u) \le 3 \epsilon t^2.$$
If $u_n$ changes sign in $\Omega$ then from Lemma \ref{L3} we can find a 
point 
$x_0 \in \Omega$ such that $u$ satisfies the hypothesis of 
Lemma~\ref{l2} 
at 
$x_0$. Thus we can perturb $v_{R,t}$ locally near $x_0$ into $\tilde v_{R,t}$ such that 
$$\EN_R (\tilde v_{R,t}) \le \EN_R (v_{R,t}) - ct^2, \quad \quad \tilde v_{R,t} \in \mathcal D_R^{Ct}(u),$$
for some $c, C >0$ depending only on $u$. In conclusion
  $$\EN_R (\tilde v_{R,t}) \le \EN_R (u) + (3 \epsilon - c)t^2, $$
  and we contradict the stability inequality if we choose $\epsilon \ll c$.
\end{proof}
  
\section{Proof of Theorem \ref{t3}}\label{44s}
  
In this section we assume that the domain $\Omega$ and the functional $F$ are invariant 
under translations in the $e_k$,..., $e_n$-directions. 

We define the notion of $u$ to be stable with respect to piecewise Lipschitz deformations in all 
directions generated by $\{e_k,...,e_n\}$ (but not with respect to 
vertical $e_{n+1}$ deformations as in Definition \ref{d4}).
Below to give a precise definition of
$\{e_k,..., e_n\}$-stability, we modify Definitions~\ref{d1} 
and~\ref{d2}
according to the following notation:

\begin{definition}
We say that $v$ is an $\{e_k,..., e_n\}$-Lipschitz deformation of $u$ in 
$B_R$
if there exist Lipschitz functions $\psi^{(k)},...,\psi^{(n)}$ with compact support in 
$B_R$, and
\begin{equation}\label{555}
\sum_{k\le i,j\le n} \|\psi^{(i)}_j\|_{L^\infty(\R^n)}^2 <1
\end{equation} such that
$$v(x)=u(x+\psi^{(k)}(x)e_k+...+\psi^{(n)}(x)e_n).$$
\end{definition}

We remark that, under condition~\eqref{555}, the map
$$x\mapsto 
x+\psi^{(k)}(x)e_k+...+\psi^{(n)}(x)e_n$$ is a diffeomorphism.

\begin{definition}\label{DD}
Let $u\in C^{0,1}(\Omega)$. We say that $v\in C^{0,1}(\Omega)$ is a 
piecewise 
$\{e_k,..., e_n\}$-Lipschitz deformation of $u$ in $B_R$
and write $$v \in D_{R,k}(u)$$
if there exist a finite number $v^{(1)}$,..., $v^{(m)}$ of $\{e_k,..., 
e_n\}$-Lipschitz 
deformations of $u$ in $B_R$ such that $$v(x)=v^{(i)}(x) \quad \quad 
\mbox{for some $i$ (depending on $x$).}$$  
Also, if all $v^{(i)}$ satisfy
$$v^{(i)}(x)=u(x+\psi^{(i,k)}(x)e_k+...+\psi^{(i,n)}(x)e_n) \quad 
\mbox{with} \quad 
\|\psi^{(i,j)}\|_{C^{0,1}(\Omega)} \le \delta$$
for some $\delta>0$, we write
$$v \in D^\delta_{R,k}(u).$$
\end{definition}

\begin{definition}\label{d5.5}
We say that $u$ is $\{e_k,..., e_n\}$-stable for $\EN$ if for 
any $R>0$ and $\epsilon>0$
there exists $\delta>0$ depending on $R$, $\epsilon$ 
and $u$ such that for all $t\in(0,\delta)$ we have that~$\EN_R(u)$ is finite and
$$ \EN_R(v)-\EN_R(u) \ge -\epsilon t^2, \qquad \forall v \in 
D^t_{R,k}(u).$$
\end{definition} 

Notice that Definitions~\ref{d5} and~\ref{d5.5} are quite different, since
vertical perturbations are allowed in Definition~\ref{d5} but not in Definition~\ref{d5.5}.
On the other hand, Definition~\ref{d5.5} allows for horizontal perturbations in $(n-k+1)$-horizontal directions, while only one horizontal direction may be perturbed in
Definition~\ref{d5}.

\begin{remark}\label{97}
{\rm We point out
that if $u \in C^1(\Omega)$  is $\{e_k,..., e_n\}$-stable and $u_k$,...,$u_n$ do not vanish all at some point then $u$ is a critical
point for $\EN$ in a neighborhood of that point
(because any vertical perturbation~$u+\epsilon\psi$ may be
written in this case as a horizontal perturbation
in the span of~$\{e_k,...,e_n\}$, due to the Implicit Function Theorem).
}\end{remark}

\begin{proof}[Proof of Theorem~\ref{t3}]
The proof of Theorem~\ref{t3}
follows as before from 
Lemma \ref{l2}, 
Remark \ref{r2} and Lemma \ref{L3}. 
First we may suppose that $k<n$, otherwise the statement is trivial. 

Let $Y_0$ be a point in $\mathcal U \subset \R^{k-1}$ and we want to show that $\tilde u$ is one-dimensional where
$$\tilde u(x_k,..,x_n):=u(Y_0,x_k,..,x_n).$$
Assume that $0 \in \R^{n-k+1}$ is such that
$\nabla \tilde u(0) $ is nonzero and it points in the $e_k$ direction. Then we may apply Lemma \ref{l2} and Remark \ref{r2} in the $e_{k+1}$,.., $e_n$ directions and conclude that $\tilde u$ is constant in a neighborhood of $0$ in all these directions. Then the set
$$\big\{ (x_{k+1},..,x_n) {\mbox{ s.t. }}
\tilde u(0,x_{k+1},..,x_n)=
\tilde u(0), \quad \nabla \tilde u(0,x_{k+1},..,x_n)=\nabla \tilde u(0) \big\}  $$
is both open and closed, hence the level set $\{ \tilde u=\tilde u(0) \}$ contains the hyperplane $0 \times \R^{n-k}$.
This argument shows that at
all points where $\nabla \tilde u$ is nonzero, the gradient must point in the $e_k$ direction, thus $\tilde u$ depends only on the $x_k$ variable. \end{proof}

\begin{remark}
{\rm
We point out that condition \eqref{H2} on $F$ can be weakened in
Theorems \ref{GENERAL2} and \ref{t3}. Since we only need~\eqref{MAJOR}
as $t \to 0$ we see from Section \ref{s3} that it suffices to have that
$$x\,\mapsto\,
\sup_{|p-\nabla u| \le |\nabla u|/2}  \quad |F_{pp}
(p,u,x')|  \, \,
|\nabla u|^2 $$
is a locally integrable function.}\end{remark}

We conclude this section with a version of Theorem~\ref{GENERAL} for $e_n$-minimiziers with
respect to piecewise Lipschitz perturbations
with norm
bounded by~$\delta$.

\begin{definition}\label{6188}
We say that~$u\in C^{0,1}(\Omega)$ is a $\{\delta,e_n\}$-minimizer 
for~$\EN$ 
if
for any $R>0$ we have
that~$\EN_R(u)$ is finite and
$$\EN_R(u) \le \EN_R(v), \quad \quad \forall v \in D_{R}^\delta(u).$$
\end{definition}

\begin{theorem}\label{GENERAL delta}
Let~$\delta>0$ and~$u \in C^1(\Omega)$ be a $\{\delta,e_n\}$-minimizer for the 
energy $\EN$ 
with $F$
satisfying~\eqref{H1} and~\eqref{H2}.

If \eqref{EE} is satisfied,
then~$u$ is monotone on each 
segment in the $e_n$-direction of length less than~$2\delta$, i.e., for
any~$\bar x\in\Omega$, either~$u_n(\bar x+te_n)\ge0$ or~$u_n(\bar 
x+te_n)\le 0$
for any~$t\in (-\delta,\delta)$.
\end{theorem}

The proof of Theorem~\ref{GENERAL delta} is identical to the one of Theorem~\ref{GENERAL}, we just need to choose~$|t|<\delta$.

\section{A one-dimensional
example}\label{55s}

In this section, we briefly discuss a one-dimensional example,
to clarify some of the notions of $e_n$-minimality  and $e_n$-stability.
We consider
\begin{equation}\label{EXA}\begin{split}
& F(p,z):=p^2-z^2,\qquad\quad\Omega:=\R \\
{\mbox{and }}& 
u(s) := \begin{cases}\cos(s+\pi/2), \quad {\mbox{ if }} s \leq -\pi/2, 
\\ \ \\ 1, \quad \quad \quad  \quad \quad {\mbox{ if }}
-\pi/2 < s <
\pi/2, \\ \ \\ 
\cos(s-\pi/2), \quad {\mbox{ if }} s\geq\pi/2.\end{cases}
\end{split}\end{equation}

\begin{proposition}\label{P.EXA}
The function~$u$ in~\eqref{EXA}
is~$e_1$-stable in~$(-\pi,\pi)$, (see 
Definition~\ref{d5.5}).
\end{proposition}

\begin{proof} Notice that~$u \in C^{1,1}(\R)\cap
C^\infty(\R\setminus\{-\pi/2,\pi/2\})$.
We prove that
\begin{equation}\label{6.2}
\begin{split}
&{\mbox{
for any $R\in(0,\pi)$, any~$\delta\in(0,\pi/2)$}}\\ &{\mbox{
and any
Lipschitz function~$\phi$ supported
in~$(-R,R)$}}\\
&{\mbox{
with $\phi\le0$ in $[-\pi/2,\pi/2]$ 
and~$\phi=0$ in $[\delta-(\pi/2),(\pi/2)-\delta]$}}\\
&{\mbox{
we have that }}
\EN_R(u+\phi)\ge \EN_R(u).
\end{split}
\end{equation}
To prove it, we may suppose~$R\in(\pi/2,\pi)$, and
we define~$I:=(-R,R)\setminus [-\pi/2,\pi/2]$, $J_-:=(-R,\,\delta-(\pi/2))$,
$J_+:=((\pi/2)-\delta,\,R)$ and~$J:=J_-\cup J_+$. Given~$\ell>0$, we also denote by~$\lambda_\ell=\pi^2/\ell^2$
the first Dirichlet eigenvalue in the interval of length~$\ell$. By taking~$\ell:=R-(\pi/2)+\delta \in (0,\pi)$,
we obtain that
$$ \int_{J_\pm}\dot\phi^2\,ds\ge \lambda_\ell \int_{J_\pm}\phi^2\,ds\ge
\int_{J_\pm}\phi^2\,ds$$
therefore
$$ \int_J \phi^2-\dot\phi^2\,ds\le 0.$$
So, we compute:
\begin{eqnarray*}
&& \EN_R(u)-\EN_R(u+\phi)
\\&&\qquad=2\int_{-R}^R u\phi-\dot u\dot\phi\,ds+\int_{-R}^R \phi^2-\dot\phi^2\,ds
\\ &&\qquad= 2\int_{-\pi/2}^{\pi/2} \phi\,ds+
2\int_{I} u\phi-\dot u\dot\phi\,ds+
\int_J  \phi^2-\dot\phi^2\,ds\\
&&\qquad\le 0 -
2\int_{I} \ddot u\phi+\dot u\dot\phi\,ds+0\\&&\qquad
= -2\int_{I} \frac{d}{ds}(\dot u\phi)\,ds\\
&&\qquad= (\dot u\phi)(-\pi/2)
-(\dot u\phi)(\pi/2) \\&&\qquad=0,
\end{eqnarray*}
which establishes~\eqref{6.2}.

Now let~$v\in D^t_{R,1}$, with~$0<t<\delta$.
Then we define~$\phi(s):=v(s)-u(s)$. Notice that~$\phi$ 
is Lipschitz, with~$\|\phi\|_{C^{0,1}(\R)}\le C\|u\|_{C^{1,1}(\R)}t$, and supported inside~$(-R,R)$. Also, $v\le 1$,
since~$v$ is a deformation of~$u$
and~$u\le1$. Therefore, for any~$s\in 
[-\pi/2,\pi/2]$, we see 
that~$\phi(s)=v(s)-1\le0$.
Finally, since~$v$ is a horizontal deformation of~$u$ of size~$t$,
we have that
$$ \inf_{[\delta-(\pi/2),(\pi/2)-\delta]} v \ge 
\inf_{[\delta-(\pi/2)-t,(\pi/2)-\delta+t]} u \ge
 \inf_{[-(\pi/2),(\pi/2)]} u =1.$$
Consequently, if~$s\in
[\delta-(\pi/2),(\pi/2)-\delta]$ we have that~$v(s)=1$ and~$\phi(s)=0$.
So we can apply~\eqref{6.2} and 
obtain~$\EN_R(v)=\EN_R(u+\phi)\ge\EN_R(u)$.
\end{proof}

As a consequence of Proposition~\ref{P.EXA},
we have that $e_n$-minimizers
are not
necessarily critical for the energy $\EN$ at the points where the gradient vanishes.
We recall that the situation for~$\{e_n, e_{n+1}\}$-stable solutions 
was different, since in that case the criticality of the
energy functional was granted by the vertical perturbations.

The example in~\eqref{EXA} may be modified in order to obtain~$\{\delta,e_1\}$-minimality
in the whole of~$\R$. For instance one may consider:
\begin{equation}\label{EXA2}\begin{split}
& F(p,z):=p^2-\max\{ z,0\}^2,\qquad\quad\Omega:=\R \\
{\mbox{and }}&
u(s) := \begin{cases}
s+\pi, \quad {\mbox{ if }} s<-\pi
\\ \ \\ 
\cos(s+\pi/2), \quad {\mbox{ if }} s \in[-\pi, -\pi/2],
\\ \ \\ 1, \quad \quad \quad  \quad \quad {\mbox{ if }}
-\pi/2 < s <
\pi/2, \\ \ \\
\cos(s-\pi/2), \quad {\mbox{ if }} s\in[\pi/2,\pi]
\\ \ \\ 
\pi-s, \quad {\mbox{ if }} s>\pi  
.\end{cases}
\end{split}\end{equation}
Then the proof of Proposition~\ref{P.EXA} may be easily modified to obtain:

\begin{proposition}\label{P.EXA2}
The function~$u$ in~\eqref{EXA2}          
is a~$\{\delta,e_1\}$-minimizer, for any~$\delta \in (0,\pi/2)$, according to
Definition~\ref{6188}.
\end{proposition}

This shows that the statement of Theorem~\ref{GENERAL delta} is optimal, since~\eqref{EXA2}
provides an example of~$\{\delta,e_1\}$-minimizer which is
monotone on intervals of length~$2\delta<\pi$
but not on intervals of larger length.

\section{Proofs of some remarks}\label{details}

\begin{proof}[Proof of Remark~\ref{R1}]
We consider a continuous function $u$
which is monotone on each line in $\R^n$.
We show that
\begin{equation}\label{62}
{\mbox{for any~$t\in\R$, the sublevel~$\{u<t\}$ is a half-space}}\end{equation}
(unless it is empty). {F}rom this, it follows that, for different values of~$t$,
$\partial\{u<t\}$ gives a collection of hyperplanes (which are parallel, since
the level sets~$\{u=t\}$ cannot intersect for different values of~$t$), and so~$u$
is one-dimensional.

To prove~\eqref{62},
first we remark that,
from the monotonicity on each line of~$u$, it follows that
\begin{equation}\label{p 62}
{\mbox{both $\{u<t\}$ and $\{u\ge t\}$ are convex sets.}}
\end{equation}
Then, we
take~$p\in \{u<t\}$. Since~$u$ is continuous, there exists~$\varrho>0$
such that
\begin{equation}\label{631}
B_\varrho(p)\subseteq \{u<t\}.
\end{equation}
We enlarge~$\varrho$ till there exists a point
\begin{equation}\label{633}
q\in\{u=t\}\cap \partial B_\varrho(p).
\end{equation}
We denote by~$\Pi_-$
the open halfspace tangent to $B_\varrho(p)$ at~$q$ that contains~$p$, and by~$\Pi_+$ the closed
halfspace tangent to $B_\varrho(p)$ at~$q$ that does not contain~$p$.

By looking at all the lines passing through~$q$,
we deduce from~\eqref{p 62},
\eqref{631}
and~\eqref{633} that
\begin{equation}\label{934}
\Pi_-\subseteq \{u<t\}\end{equation}
\begin{comment}
Indeed, suppose, by contradiction, that there exists a point
\begin{equation}\label{634}
a_-\in\Pi_-\cap \{ u\ge t\},\end{equation}
and let~$\ell_-$ be the segment~$a_-$ and~$q$. By~\eqref{633}, \eqref{634}
and~\eqref{p 62}, we have that
\begin{equation}\label{6631}
\ell_- \subseteq \{u\ge t\}.
\end{equation}
On the other hand,~$\ell_-$
must intersect~$B_\varrho(p)$, i.e. there exists~$b_-\in \ell_-\cap B_\varrho(p)$.
{F}rom~\eqref{631} we obtain that~$b_-\in \{u<t\}$, while from~\eqref{6631}
that~$b_-\in\{u\ge t\}$. This contradiction proves~\eqref{934}.
\end{comment}
and
\begin{equation}\label{935}
\Pi_+\subseteq \{u\ge t\}.\end{equation}
\begin{comment}
Suppose, by contradiction, that there exists
a point
\begin{equation}\label{2634}
a_+\in\Pi_+\cap \{ u<t\}.\end{equation}
By continuity, a neighborhood of~$a_+$ also lies in~$\{ u<t\}$, therefore,
without loss of generality, we may suppose that
\begin{equation}\label{i}
{\mbox{$a_+$ must lie in the interior of~$\Pi_+$. }}\end{equation}
Now, let~$b_+$ be the symmetric point of~$a_+$ with respect to~$q$.
That is, we consider the segment~$\ell_+$ which joins~$a_+$ and~$b_+$
and has~$q$ as middle point. {F}rom~\eqref{i}, we see that~$b_+\in\Pi_-$, and therefore,
by~\eqref{934}, we obtain that~$b_+\in \{u<t\}$. This,~\eqref{2634}
and~\eqref{p 62} imply that~$\ell_+ \subseteq \{u<t\}$. Since~$\ell_+\ni q$,
we conclude that~$q\in \{u<t\}$, which is in contradiction with~\eqref{633}.

This establishes~\eqref{935}.\end{comment}
By taking the complementary sets in~\eqref{935} and noticing that~$\Pi_+$ is the complement of~$\Pi_-$,
we conclude that
$$ \Pi_-\supseteq \{u<t\}.$$
This and~\eqref{934} give that~$\Pi_-=\{u<t\}$, proving~\eqref{62}.
\end{proof}

\begin{proof}[Proof of Remark~\ref{stable}]
Suppose that~$u\in C^2(\R^n)$ is a classical stable solution, i.e.
a critical point of the energy functional satisfying~\eqref{stable eq}.
Since $F\in C^2$, we have that for any Lipschitz function~$\phi$ supported in a given ball~$B_R$,
\begin{equation*}\label{922}\begin{split}
&F(\nabla u+t\nabla\phi,u+t\phi,x)-F(\nabla u,u,x)
\\ =\, &t\Big( F_{p_i}\phi_i+F_z\phi\Big)+\frac{t^2}{2}
\Big(  F_{p_i p_j}  \phi_i \phi_j+F_{zz}\phi^2+2F_{p_i z}\phi\phi_i\Big)
+o(t^2),
\end{split}\end{equation*}
where the derivatives of $F$ are evaluated at~$(\nabla u,u,x)$.
Notice that~$o(t^2)$ above
only depends on the Lipschitz norms of~$\phi$ and~$ u$ in~$B_R$, and on the $C^2$-norm of~$F$
in a bounded set (depending on~$R$ as well).
When we integrate the equality above over~$B_R$, the term of order~$t$ disappears since~$u$
is a critical point, therefore we obtain
\begin{equation}\label{92}\begin{split}
&\EN_R(u+t\phi)-\EN_R(u)\\&\qquad=\frac{t^2}2\int_{B_R}
\Big(  F_{p_i p_j} (\zeta) \phi_i \phi_j+F_{zz}(\zeta)\phi^2+2F_{p_i z}(\zeta)\phi\phi_i\Big)\,dx
+o(t^2).\end{split}\end{equation}
Dividing by~$t^2$ and recalling~\eqref{stable eq}, we conclude that
$$ \int_{B_R}
\Big(  F_{p_i p_j} (\zeta) \phi_i \phi_j+F_{zz}(\zeta)\phi^2+2F_{p_i z}(\zeta)\phi\phi_i\Big)\,dx\ge 0.$$
Hence, going back to~\eqref{92}, we obtain that
\begin{equation}\label{1555}
\EN_R(u+t\phi)-\EN_R(u)\ge o(t^2).\end{equation}
Now, given~$w \in\mathcal D^t_R(u)$, we take~$\phi:=(w-u)/t$.
Notice that the Lipschitz norm of~$\phi$ is bounded uniformly in~$t$, therefore~\eqref{1555}
implies~\eqref{55} and so~$u$ is $\{e_n,e_{n+1}\}$-stable.

Viceversa, suppose that~$u$ is $\{e_n,e_{n+1}\}$-stable. Then $u$ is a critical point and~\eqref{55} implies~\eqref{stable eq} by choosing~$w:=u+t\phi$
and taking~$\epsilon$ arbitrarily small. This shows that $u$ is a stable
solution.
\end{proof}

\begin{proof}[Proof of Remark \ref{log}]
We define $e_0(s):=s$ and then recursively $$ e_k(s):=\exp (e_{k-1}(s))=
\underbrace{\exp\circ\dots\circ\exp}_{\text{$k$ times}} s$$
for any $k\in\N$, $k\ge1$. Let $\theta_k(R):=e_{k}(\sqrt{\ell_k(R)})$.
Notice that
\begin{equation}\label{21} \ell_{k+1}(\theta_k(R))=\log\sqrt{\ell_k(R)}=\frac12 \log(\ell_k(R))=
\frac{\ell_{k+1}(R)}{2}.\end{equation}
By induction over $k$, one sees that
\begin{equation}\label{56}
\ell_k'(r)=\big( \pi_{k-1}(r)\big)^{-1},\end{equation}
where the notation in \eqref{PI} was used
together with the setting $\pi_{-1}(r):= 1$ (in this way, $\pi_{k}(r)=\ell_k(r)
\pi_{k-1}(r)$ for any $k\in\N$).
We 
obtain that
\begin{eqnarray*}
\pi_k'(r)&=& \sum_{m=0}^k \;\prod_{
\genfrac{}{}{0pt}{}{0\le j\le k}{j\neq m} }\ell_j(r)\ell_m'(r)
\\ &=&
\sum_{m=0}^k \;\prod_{j=m+1}^k \ell_j(r)\\ &\le&
(k+1) \prod_{j=1}^k \ell_j(r) \\
&=&(k+1) r^{-1}\pi_k(r)\end{eqnarray*}
for large $r$, and so
\begin{equation}\label{7.8}
-\frac{d}{dr} \big( \pi_k(r)\big)^{-2}=
2\big( \pi_k(r)\big)^{-3}\pi_k'(r)
\le 2(k+1)
r^{-1}\big( \pi_k(r)\big)^{-2}
\end{equation}
for large $r$.
Now, recalling~\eqref{21}, we modify \eqref{12.1} as follows:
\begin{equation}\label{12.1a}
\psi_R(s):=\left\{
\begin{matrix}
1, & {\mbox{ if }}0\le s\le \theta_k(R), \\
\, \\
2-\displaystyle\frac{2\ell_{k+1}(s)}{\ell_{k+1}(R)}, & {\mbox{ if }} \theta_k(R)<s\le R,\\ \, \\
0, & {\mbox{ if }}s>R.
\end{matrix}
\right.\end{equation}
{F}rom \eqref{12.1a} and \eqref{56} we see that
\begin{equation}\label{12.2a}
\psi'_R(s)=\left\{
\begin{matrix}
-\displaystyle\frac{2}{\ell_{k+1}(R) \pi_k(s)}, & {\mbox{ if }} \theta_k(R)<s\le R,\\ \, \\
0, & {\mbox{ otherwise}}.
\end{matrix}\right.
\end{equation}
Notice that \eqref{12.1a} and \eqref{12.2a} reduce to \eqref{12.1} and \eqref{12.2}
respectively when $k=0$.
Then, we can argue as in Section \ref{s3}.
In this case, \eqref{AB} gets replaced by
\begin{equation}\label{7.7}\begin{split}
& \EN_R(u^+_{R,t})+\EN_R(u^-_{R,t})-2\EN_R(u)\\ &\qquad
\le C\frac{t^2}{\big(\ell_{k+1}(R)\big)^2}
\int_{\Omega\cap(B_R\setminus B_{\theta_k(R)})}\frac{|F_{pp}(\nabla u,u,x')|\,
|\nabla u|^2}{\big( \pi_{k}(|x|)\big)^2}\,dx\\
&\qquad
= C\frac{t^2}{\big(\ell_{k+1}(R)\big)^2}
\int_{\Omega\cap(B_R\setminus B_{\theta_k(R)})} |F_{pp}(\nabla u,u,x')|\,
|\nabla u|^2\, \sigma(|x|)\,dx,
\end{split}\end{equation}
where
$$ \sigma(r):=
\big( \pi_{k}(r)\big)^{-2}.$$
Therefore, we recall \eqref{5.2a}
and we notice that, in this case, $a(r)\le Cr \pi_k(r)$ for large $r$, thanks
to \eqref{EE weak}. So we use \eqref{7.8} and~\eqref{56},
and, instead of \eqref{5.2b}, in this case we bound the last integral on
right hand side of \eqref{7.7} in polar coordinates by
\begin{eqnarray*}
&& \int_{\theta_n(R)}^R a'(r)\sigma(r)\,dr\le a(R)\sigma(R)-
\int_{\theta_n(R)}^R a(r)\sigma'(r)\,dr
\\ &&\qquad\le CR \big(\pi_k(R)\big)^{-1}+C\int_{\theta_n(R)}^R
\big(\pi_k(r)\big)^{-1} \,dr\\
&&\qquad= CR \big(\pi_k(R)\big)^{-1}+C\int_{\theta_n(R)}^R
\ell_{k+1}'(r)\,dr\\
&&\qquad\le C+C\ell_{k+1}(R).
\end{eqnarray*}
Therefore, \eqref{7.7} gives in this case
\begin{eqnarray*}&&
\limsup_{R\rightarrow+\infty}\sup_{t\in(0,\theta_k(R)/4)}
t^{-2}\Big(                  
\EN_R(u^+_{R,t})+\EN_R(u^-_{R,t})-2\EN_R(u)\Big)\\ &&\qquad\le
\limsup_{R\rightarrow+\infty}\sup_{t\in(0,\theta_k(R)/4)}
\frac{C}{\big( \ell_{k+1}(R)\big)^2} \,\big( 1+\ell_{k+1}(R)\big)
=0,\end{eqnarray*}
which replaces~\eqref{MAJOR}
in this case.
\end{proof}

\begin{proof}[Proof of Remark~\ref{ABS}] Here we
construct a one-dimensional example of a Lipschitz
function $u:\R\rightarrow\R$ that satisfies \eqref{EE}
and that minimizes the energy with respect to any $e_n$-Lipschitz
deformation, without being monotone. For this we take $u(t):=|t|$,
$\Omega:=\R$ and $F:=|p|^2$. Then, \eqref{EE} is obvious,
and clearly $u$ is not monotone. Let us check that it is
minimal with respect to any $e_n$-Lipschitz
deformation, as described in Definition \ref{d1}: for this
let $\psi$ be Lipschitz and supported in $(-R,R)$,
with $|\psi'|<1$, and $v(t)=u(t+\psi(t))=
|t+\psi(t)|$. We have $$ |v'(t)|^2-|u'(t)|^2=(1+\psi'(t))^2-1=
2\psi'(t)+(\psi'(t))^2$$
for almost any $t\in(-R,R)$. Therefore, if we integrate over $(-R,R)$
and we use that $\psi(-R)=0=\psi(R)$, we obtain
$$ \EN_R(v)-\EN_R(u)=\int_{-R}^R (\psi'(t))^2\,dt\ge 0,$$
which is the minimality with respect to $e_n$-Lipschitz
deformations.

It is worth noticing that $u$ is not an $e_n$-minimizer,
since piecewise $e_n$-Lipschitz
deformations may decrease the energy (this justifies the importance
of Definition \ref{d2}). To show this, we take $R:=2$,
\begin{eqnarray*} &&\psi^{(1)}(t):=\left\{
\begin{matrix}
-\displaystyle\frac{2+t}{3}, & {\mbox{ if }} t\in [-2,1],\\
\,\\
t-2 , & {\mbox{ if }} t\in (1,2],
\end{matrix}\right.\\
{\mbox{ and }}&&
\psi^{(2)}(t):=\left\{
\begin{matrix}
t+2 , & {\mbox{ if }} t\in [-2,-1],\\ \, \\
\displaystyle\frac{2-t}{3}, & {\mbox{ if }} t\in (-1,2].
\end{matrix}\right.\end{eqnarray*}
Let also $v^{(i)}(t):=u(t+\psi^{(i)}(t))$ and
$$ v(t):=\left\{
\begin{matrix}
v^{(1)}(t), & {\mbox{ if }} t\in [-2,0],\\
v^{(2)}(t), & {\mbox{ if }} t\in (0,2].
\end{matrix}\right.$$
Then $v$ is a piecewise $e_n$-Lipschitz
deformation of $u$ according to Definition \ref{d2} and one
may explicitly compute that
$$ v(t)=\frac{2(|t|+1)}{3}.$$
In particular, $\EN_R(v)=(8/9)R<2R=\EN_R(u)$, which shows that
$u$ is not an $e_n$-minimizer.
\end{proof}

\begin{proof}[Proof of Remark~\ref{6fsv}]
We define~$a(t)$, $\lambda_i(t)$, $\Lambda_i(t)$ and~$A_{ij}(p)$ as 
in~\cite{FSV} (see, in particular, formulas~(1.4)--(1.6) there).
To avoid confusion with the notation here, the function~$F$ introduced 
below~(1.6) in~\cite{FSV} will be denoted by~$\tilde F$.
The goal is to
apply Theorem~\ref{GENERAL2}
with~$F(p,z,x):=\Lambda_2(|p|)+\tilde F(z)$
(since this and Remark~\ref{R1} here plainly imply
Theorems~1.1 and~1.2 in~\cite{FSV}). For this, we need
to check the convexity of~$F$ in~$p$ and conditions~\eqref{H2}
and~\eqref{EE}.
We may focus on the case~$n=3$,
i.e. on the case of Theorem~1.2 of~\cite{FSV} (this allows us
to take also assumptions~(B1) and~(B2) in~\cite{FSV}). {F}rom~(1.6) 
and~(1.5) in~\cite{FSV}, we see that
$$ F_{p_i}(p,z,x)=\lambda_2(|p|)\,p_i=a(|p|)\,p_i$$
and so
$$ F_{p_ip_j}(p,z,x)=a(|p|)\,\delta_{ij} +a'(|p|)\,|p|^{-1}p_i 
p_j=A_{ij}(p).$$
Therefore, Lemma~2.1 in~\cite{FSV} gives the desired convexity of~$F$
and it implies that
\begin{equation}\label{FF1}
{\mbox{$
|F_{pp}(p,z,x)|$ is bounded from above and below by $
C\big(\lambda_1(|p|)+\lambda_2(|p|)\big)$}}.\end{equation}
On the other hand, by Lemma~4.2 of~\cite{FSV}, we have that, if~$|p|\le 
M$, then
\begin{equation}\label{FF2}
\lambda_1(|p|)\le C_M \lambda_2(|p|)\end{equation}
and
\begin{equation}\label{FF3}
\lambda_2(|p+q|)\le C_M \lambda_2(|p|)\end{equation}
if~$|q|\le |p|/2$,
for
suitable~$C_M>0$ (possibly varying line after line). By 
plugging~\eqref{FF2} into~\eqref{FF1} 
we obtain that, if~$|p|\le M$,
\begin{equation}\label{FF5}
|F_{pp}(p,z,x)|\le
C_M \lambda_2(|p|).\end{equation}
Using~\eqref{FF5} and~\eqref{FF3}
we see that if~$2|q|\le |p|\le M$,
$$ |F_{pp}(p+q,z,x)|\le C_M  \lambda_2(|p+q|)\le C_M \lambda_2(|p|)
\le C_M|F_{pp}(p,z,x)|,$$
which gives~\eqref{H2} (notice that we may suppose~$|p|=|\nabla u|\le M$
in this case).

Moreover, using~\eqref{FF5} here and~(4.3) of~\cite{FSV}, we obtain
$$ |F_{pp}(p,z,x)|\,|p|^2\le C_M \lambda_2(|p|)\,|p|^2=
C_M a(|p|)\,|p|^2\le \Lambda_2(|p|).$$
This and~(5.16) in~\cite{FSV} imply
$$ \int_{B_R} |F_{pp}(\nabla u,u,x)|\,dx\le C_M 
\int_{B_R}\Lambda_2(|\nabla u|)\,dx\le C_M R^2.$$
This shows that~\eqref{EE} holds true in this case: so we may use
Theorem~\ref{GENERAL2}, then recall
Remark~\ref{R1}, and obtain the one-dimensional results of Theorems~1.1 
and~1.2 of~\cite{FSV}.
\end{proof}

\end{document}